\documentclass[11pt,a4paper]{amsart}
\usepackage{graphicx}
\usepackage{amssymb,latexsym,amsmath}
\usepackage{amsfonts,amsbsy,bm}
\usepackage{amscd}
\newtheorem{theorem}[subsection]{Theorem}
\newtheorem{proposition}[subsection]{Proposition}

\theoremstyle{definition}
\newtheorem{definition}[subsection]{Definition}
\newtheorem{remark}[subsection]{Remark}

\newcommand{\Ba}[1]{\begin{array}{#1}}
\newcommand{\Ea}{\end{array}}
\newcommand{\Be}{\begin{equation}}
\newcommand{\Ee}{\end{equation}}
\newcommand{\Bea}{\begin{eqnarray}}
\newcommand{\Eea}{\end{eqnarray}}
\newcommand{\Beas}{\begin{eqnarray*}}
\newcommand{\Eeas}{\end{eqnarray*}}
\newcommand{\Benu}{\begin{enumerate}}
\newcommand{\Eenu}{\end{enumerate}}
\newcommand{\Bi}{\begin{itemize}}
\newcommand{\Ei}{\end{itemize}}
\renewcommand {\phi} {{\varphi}}

\begin{document}

\title[Martingale Hardy-Amalgam Spaces]{Martingale Hardy-Amalgam Spaces: Atomic decompositions and Duality
}


\author{Justice Sam Bansah and Beno\^it F. Sehba}
\address{Department of Mathematics, University of Ghana,\\ P. O. Box LG 62 Legon, Accra, Ghana}
\email{fjaccobian@gmail.com}
\address{Department of Mathematics, University of Ghana,\\ P. O. Box LG 62 Legon, Accra, Ghana}
\email{bfsehba@ug.edu.gh}


\begin{abstract}
In this paper, we introduce the notion of martingale Hardy-amalgam spaces: $ H^s_{p,q},\,\,\mathcal{Q}_{p,q}$ and $\mathcal{P}_{p,q}$. We present two atomic decompositions for these spaces. The dual space of $H^s_{p,q}$ for $0<p\le q\le 1$ is shown to be a Campanato-type space.

\end{abstract}


\maketitle 

\section{Introduction}
We owe the martingale theory to J. L. Doob from his seminal work \cite{Doob}. The theory was later developed by D. L. Burkholder, A. M. Garsia, R. Cairoli, B. J. Davis and their collaborators (see \cite{Burkholder1,Burkholder2,Cairoli,Davis,Garsia,Neveu,Weisz} and the references therein). Martingales are particularly interesting because of their connection and applications in Fourier analysis, complex analysis and classical Hardy spaces (see for example \cite{Bass,Burkholder1,Doob,Durett,Long,Weisz}).
\vskip .1cm
The classical martingale Hardy spaces $H_p$ are defined as the spaces of martingales whose maximal function, quadratic variation or conditional quadratic variation belongs to the usual Lebesgue spaces $L_p$ with a probability measure. The atomic decompositions, martingale embeddings and dual spaces of these spaces and related spaces are discussed by F. Weisz in \cite{Weisz}. This type of studies has been considered by several authors for some generalizations of the classical Lebesgue spaces as Lorentz spaces, Orlicz spaces, Orlicz-Musielak spaces...(see \cite{Ho,Jiao,Miyamoto,Ren,Xie1,Xie,Yong}).
\vskip .1cm
In this paper, inspired by the recent introduction of Hardy-amalgam spaces in the classical Harmonic analysis (\cite{Justin1,Justin2,Zhang}), we replace Lebesgue spaces in the definition of the classical martingale Hardy spaces by the Wiener amalgam spaces, introducing then the notion of martingale Hardy- amalgam spaces. As Wiener amalgam spaces generalize Lebesgue spaces, martingale Hardy-amalgam spaces then generalize the martingale Hardy spaces presented in \cite{Weisz}. We provide atomic decompositions and we characterize the dual spaces of these martingale Hardy-amalgam spaces and their associated spaces of predictive martingales and martingales with predictive quadratic variation.
\vskip .1cm
{\bf Motivation}: We are motivated essentially by two observations. As fisrt observation we note that in the case of classical Hardy-amalgam spaces of \cite{Justin1}, atomic decomposition is obtained only for the range $0<p\le q\le 1$. The question that came into our mind was to know if any answer could be given beyond this range. The second observation is that dyadic analogues of Hardy spaces are pretty practical when comes to the study of operators on these spaces. It was then natural to consider dyadic analogues and more generally, martingale analogues of the Hardy-Amalgam spaces of \cite{Justin1}. We also note that unlike \cite{Justin1}, our universe of discourse is an arbitrary non-empty set.
\section{Preliminaries: Menagerie of spaces}
We introduce here some function spaces in relation with our concern in this paper.
\subsection{Wiener Amalgam Spaces}
Let $\Omega$ be an arbitrary non-empty set and let $\{\Omega_j\}_{j\in\mathbb{Z}}$ be a sequence of nonempty subsets of $\Omega$ such that $\Omega_j\cap\Omega_i=\emptyset$ for $j\ne i,$ and $$\bigcup_{j\in\mathbb{Z}}\Omega_j=\Omega.$$ 
For $0<p,q\le \infty$, the classical amalgam of $L_p$ and $l _q,$ denoted $L_{p,q},$ on $\Omega$ consists of functions which are locally in $L_p$ and have $l_q$ behaviour (c.f [10]), in the sense that the $L_p$-norms over the subsets $\Omega_j\subset\Omega$ form an $l_q-$sequence i.e. for $p,q\in(0,\infty),$ $$L_{p,q}=\left\{f : \|f\|_{p,q}:=\|f\|_{L_{p,q}(\Omega)}<\infty \right\}$$
 where \begin{eqnarray}\label{amag1}\|f\|_{p,q}:=\|f\|_{L_{p,q}(\Omega)}:= \left[\sum_{j\in\mathbb{Z}}\left(\int_{\Omega}|f|^p\mathbf{1}_{\Omega_j}\mathrm{d}\mathbb{P}\right)^{\frac{q}{p}}\right]^{\frac{1}{q}},\end{eqnarray}  for $0<q<\infty$, and for $q=\infty,$ $$L_{p,\infty}=\left\{f :\, \|f\|_{p,\infty}:=\|f\|_{L_{p\infty}(\Omega)}:= \sup_{j\in\mathbb{Z}}\left(\int_{\Omega}|f|^p\mathbf{1}_{\Omega_j}\mathrm{d}\mathbb{P}\right)^{\frac{1}{p}}<\infty\right\}.$$ 
As usual, $\mathbf{1}_A$ is the indicator function of the set $A$.
We observe the following:
\begin{itemize}
\item Endowed with the (quasi)-norm $\|\cdot\|_{p,q}$, the amalgam space $L_{p,q}$ is a complete space, and a Banach space for $1\le p,q\le \infty$.
\item $\|f\|_{p,p} = \|f\|_p$ for $f\in L_p(\Omega).$
\item $\|f\|_{p,q}\le \|f\|_{p}$ if $p\le q$ and $f\in L_p(\Omega)$.
\item $\|f\|_{p}\le \|f\|_{p,q}$ if $q\le p$ and $f\in L_{p,q}(\Omega)$.
\end{itemize} 
Amalgam function spaces have been essentially considered in the case $\Omega=\mathbb{R}^d,\,d\in\mathbb{N}$, and in the case $d=1$, the subsets $\Omega_j$ are just the intervals $[j,j+1)$, $j\in \mathbb{Z}$. It is also known that different appropriate choices of the sequence of sets $(\Omega_j)_{j\in \mathbb{Z}}$ provide the same spaces (see for example \cite{Justin1,Heil}). For more on amalgam spaces, we refer the reader to \cite{Fournier,Holland}.

\subsection{Martingale Hardy Spaces via Amalgams}
In the remaining of this text, all the spaces are defined with respect to the same probability space $(\Omega,\mathcal{F},\mathbb{P}).$ Let $(\mathcal{F}_n)_{n\ge 0}:=(\mathcal{F}_n)_{n\in \mathbb{Z}_+}$ be a non-decreasing sequence of $\sigma$-algebra with respect to the complete ordering on $\mathbb{Z}_+=\{0,1,2\cdots\}$ such that $$\sigma\left(\bigcup_{n\in\mathbb{N}}\mathcal{F}_n\right)=\mathcal{F}.$$ For $n\in\mathbb{Z}_+,$ the expectation operator and the conditional expectation operator relatively to $\mathcal{F}_n$ are denoted by $\mathbb{E}$ and $\mathbb{E}_{n}$ respectively. We denote by $\mathcal{M}$ the set of all martingales $f=(f_n)_{n\ge 0}$ relatively to the filtration $(\mathcal{F}_n)_{n\geq 0}$ such that $f_0=0$. We recall that for $f\in \mathcal{M}$, its martingale difference is denoted $d_nf=f_n-f_{n-1}$, $n\geq 0$ with the convention that $d_0f=0$.
\vskip .1cm
We recall that the stochastic basis, $(\mathcal{F}_n)_{n\ge 0}$ is said to be regular if there exists $R>0$ such that $f_n\le Rf_{n-1}$ for all non-negative martingale $(f_n)_{n\geq 0}$.
\vskip .1cm
A martingale $f=(f_n)_{n\geq 0}$  is said to be $L_p$ bounded $(0<p\le\infty)$ if $f_n\in L_p$ for all $n\in \mathbb{Z}_+$ and we define $$\|f\|_p:=\sup_{n\in\mathbb{N}}\|f_n\|_p <\infty.$$ We recall that  $$\|f\|_p= (\mathbb{E}(|f|^p))^{\frac{1}{p}}=\left(\int_{\Omega}|f|^p\mathrm{d}\mathbb{P}\right)^{\frac{1}{p}}.$$ 
\vskip .1cm
We denote by $\mathcal{T}$ the set of all stopping times on $\Omega$.
For $\nu\in \mathcal{T}$, and $(f_n)_{n\ge 0}$ an integrable sequence, we recall that the associated stopped sequence $f^{\nu}=(f^{\nu}_n)_{n\ge 0}$ is defined by $$f^{\nu}_n=f_{n\wedge\nu},\quad n\in\mathbb{Z}_+.$$ 
\vskip .1cm
For a martingale $f=(f_n)_{n\geq 0}$,  the quadratic variation, $S(f),$ and the conditional quadratic variation, $s(f),$ of $f$ are defined by $$S(f) = \left(\sum_{n\in\mathbb{N}}|d_nf|^2 \right)^{\frac{1}{2}}\quad\mbox{and}\quad s(f) = \left(\sum_{n\in\mathbb{N}}\mathbb{E}_{n-1}|d_nf|^2 \right)^{\frac{1}{2}}$$ respectively. We shall agree on the notation $$S_n(f) = \left(\sum_{i=1}^n|d_if|^2 \right)^{\frac{1}{2}}\quad\mbox{and}\quad s_n(f) = \left(\sum_{i=1}^n\mathbb{E}_{i-1}|d_if|^2 \right)^{\frac{1}{2}}.$$ The maximal function $f^*$ or $M(f)$ of the martingale $f$ is defined by $$M(f)=f^*:=\sup_{n\in\mathbb{N}}|f_n|.$$ 
\vskip .2cm
We now introduce the martingale Hardy-amalgam spaces:  $H^s_{p,q},\,\, \mathcal{Q}_{p,q}$ and $\mathcal{P}_{p,q}.$ Let $0<p<\infty$ and $0<q\leq \infty$. The first space is defined as follows. 
\begin{itemize}
\item[i.]$H^s_{p,q}(\Omega) = \{f\in\mathcal{M} : \|f\|_{H^s_{p,q}(\Omega)} :=\|s(f)\|_{p,q}<\infty\}$.
\end{itemize}
Let $\Gamma$ be the set of all sequences $\beta=(\beta_n)_{n\ge 0}$ of adapted (i.e. $\beta_n$ is $\mathcal{F}_n$-measurable for any $n\in \mathbb{Z}_+$) non-decreasing, non-negative  functions and define $$\beta_{\infty}:=\lim_{n\rightarrow\infty}\beta_n.$$
\begin{itemize}
\item[ii.] The space $\mathcal{Q}_{p,q}(\Omega)$ consists of all martingales $f$  for which there is a  sequence of functions $\beta=(\beta_n)_{n\ge 0}\in \Gamma$ such that $S_n(f)\le\beta_{n-1}$ and $\beta_\infty\in L_{p,q}(\Omega)$. 

We endow  $\mathcal{Q}_{p,q}(\Omega)$ with 
$$\|f\|_{\mathcal{Q}_{p,q}(\Omega)}:=\inf_{\beta\in\Gamma}\|\beta_{\infty}\|_{p,q}.$$
\item[iii.] The space $\mathcal{P}_{p,q}(\Omega)$ consists of all martingales $f$  for which there is a sequence of functions $\beta=(\beta_n)_{n\ge 0}\in \Gamma$ such that $|f_n|\le\beta_{n-1}$ and $\beta_\infty\in L_{p,q}(\Omega)$. 

We endow $\mathcal{P}_{p,q}(\Omega)$ with
$$\|f\|_{\mathcal{P}_{p,q}(\Omega)}:=\inf_{\beta\in\Gamma}\|\beta_{\infty}\|_{p,q}.$$
\end{itemize}
A martingale $f\in\mathcal{P}_{p,q}(\Omega)$ is called predictive martingale and a martingale $f\in\mathcal{Q}_{p,q}(\Omega)$ is a martingale with predictive quadratic variation.
\vskip .1cm
In the sequel, when there is no ambiguity, the spaces $H^s_{p,q}(\Omega),\mathcal{Q}_{p,q}(\Omega)$ and $\mathcal{P}_{p,q}(\Omega)$ will be just denoted $H^s_{p,q},\,\, \mathcal{Q}_{p,q}$ and $\mathcal{P}_{p,q}$ respectively. The same will be done for the associted (quasi)-norms.
\begin{remark} 
\begin{itemize}
\item Observe that when $0<p=q<\infty$, the above spaces are just the spaces $H^s_{p},\,\, \mathcal{Q}_{p}$ and $\mathcal{P}_{p}$ defined and studied in \cite{Weisz}. 
\item Hardy-amalgam spaces of classical functions $H_{p,q}$ of $\mathbb{R}^d$ ($d\geq 1$) were introduced recently by V. P. Abl\'e and J. Feuto in \cite{Justin1} where they provided an atomic decomposition for these spaces for $0<p,q\le 1$. In \cite{Justin2}, they also characterized the corresponding dual spaces, for  $0<p\le q\le 1$. A generalization of their definition and their work was pretty recently obtained in \cite{Zhang}.
\item Our definitions here are inspired from the work \cite{Justin1} and the usual definition of martingale Hardy spaces. 
\end{itemize}
\end{remark}
\section{Presentation of the results}
We start by defining the notion of atoms.
\begin{definition}\label{atdef1}
Let $0<p<\infty$, and $\max(p,1)<r\le \infty$. A measurable function $a$ is a $(p,r)^s$-atom (resp. $(p,r)^S$-atom, $(p,r)^*$-atom) if there exists a stopping time $\nu\in \mathcal{T}$ such that \begin{itemize}\item[(a1)] $a_n:=\mathbb{E}_na=0$ if $\nu\ge n$; \item[(a2)]$\|s(a)\|_{r,r}:=\|s(a)\|_{r}\,(\textrm{resp.} \|S(f)\|_{r},\|a^*\|_{r})\le\mathbb{P}(B_{\nu})^{\frac 1r-\frac{1}{p}}$ \end{itemize}where $B_{\nu} = \{\nu\neq \infty\}.$ 
\end{definition}
We also have the following other definition of an atom.
\begin{definition}\label{atdef2}
Let $0<p<\infty$, $0<q\le \infty$ and $\max(p,1)<r\leq\infty$. A measurable function $a$ is a $(p,q,r)^s$-atom (resp. $(p,q,r)^S$-atom, $(p,q,r)^*$-atom) if there exists a stopping time $\nu\in \mathcal{T}$ such that condition (a1) in Definition \ref{atdef1} is satisfied and \begin{itemize} \item[(a3)]$\|s(a)\|_{r}\, (\textrm{resp.} \|S(a)\|_{r}, \|a^*\|_{r})\le\mathbb{P}(B_{\nu})^{\frac{1}{r}}\|1_{B_\nu}\|_{p,q}^{-1}.$ \end{itemize} 
\end{definition}
\vskip .2cm
We denote by $\mathcal{A}(p,q,r)^s$ (resp. $\mathcal{A}(p,q,r)^S$, $\mathcal{A}(p,q,r)^*$) the set of all sequences of triplets $(\lambda_k,a^k,\nu^k)$, where $\lambda_k$ are nonnegative numbers, $a^k$ are $(p,r)^s$-atoms (resp. $(p,r)^S$-atoms, $(p,r)^*$-atoms) and $\nu^k\in \mathcal{T}$ satisfying conditions ($a1$) and ($a2$) in Definition \ref{atdef1} and such that for any $0<\eta\le 1$, $$\sum_{k}\left(\frac{\lambda_k}{\mathbb{P}(B_{\nu^k})^{\frac{1}{p}}}\right)^{\eta}\mathbf{1}_{B_{\nu^k}}\in L_{\frac p\eta,\frac q\eta}.$$
\vskip .1cm
We denote by $\mathcal{B}(p,q,r)^s$ (resp. $\mathcal{B}(p,q,r)^S$, $\mathcal{B}(p,q,r)^*$) the set of all sequences of triplets $(\lambda_k,a^k,\nu^k)$, where $\lambda_k$ are nonnegative numbers, $a^k$ are $(p,q,r)^s$-atoms (resp. $(p,q,r)^S$-atoms, $(p,q,r)^*$-atoms) and $\nu^k\in \mathcal{T}$ satisfying conditions ($a1$) and ($a3$) in Definition \ref{atdef2} and such that for any $0<\eta\le 1$, $$\sum_{k}\left(\frac{\lambda_k}{\|1_{B_{\nu^k}}\|_{p,q}}\right)^{\eta}\mathbf{1}_{B_{\nu^k}}\in L_{\frac p\eta,\frac q\eta}.$$
\vskip .1cm
We observe that $\mathcal{A}(p,q,r)^s\subseteq \mathcal{B}(p,q,r)^s$ if $p\le q$ and $\mathcal{B}(p,q,r)^s\subseteq \mathcal{A}(p,q,r)^s$ if $q\le p$. The same relation holds between the other sets of triplets.

Our first atomic decomposition of the spaces $H^s_{p,q}$ is as follows.
\begin{theorem}\label{thm:at1}
Let $0<p<\infty$, $0<q\leq \infty$ and let $\max(p,1)<r\leq\infty$. If the martingale $f\in \mathcal{M}$ is in $H^s_{p,q}$, then there exists a sequence of triplets $(\lambda_k,a^k,\nu^k)\in \mathcal{A}(p,q,r)^s $  such that for all $n\in\mathbb{N}$, \begin{eqnarray}\label{eq:at11}\sum_{k\in\mathbb{Z}}\lambda_k\mathbb{E}_na^k=f_n \end{eqnarray} and for any $0<\eta\le 1$, \begin{equation}\label{eq:at12}\left\|\sum_{k\ge 0}\left(\frac{\lambda_k}{\mathbb{P}(B_{\nu^k})^{\frac{1}{p}}}\right)^{\eta}\mathbf{1}_{B_{\nu^k}}\right\|^{\frac{1}{\eta}}_{\frac{p}{\eta},\frac{q}{\eta}}\le C\|f\|_{H^s_{p,q}}.\end{equation} Moreover, $$\sum_{k=l}^m\lambda_ka^k\longrightarrow f$$ in $H^s_{p,q}$  as $m\rightarrow\infty,\,\,l\rightarrow -\infty$. 
\vskip .1cm
Conversely if $f\in \mathcal{M}$ has a decomposition as in (\ref{eq:at11}), then for any $0<\eta\le 1$, $$\|f\|_{H^s_{p,q}}  \le C\left\|\sum_{k\in\mathbb{Z}}\left(\frac{\lambda_k}{\mathbb{P}(B_{\nu^k})^{\frac{1}{p}}}\right)^{\eta}\mathbf{1}_{B_{\nu^k}}\right\|^{\frac{1}{\eta}}_{\frac{p}{\eta},\frac{q}{\eta}}.$$
\end{theorem}
Using our second definition of atoms, we also obtain the following atomic decomposition.
\begin{theorem}\label{thm:at2}
Let $0<p<\infty$, $0<q\leq \infty$ and let $\max(p,1)<r\leq\infty$. If the martingale $f\in \mathcal{M}$ is in $H^s_{p,q}$ then there exists a sequence of triplets $(\lambda_k,a^k,\nu^k)\in \mathcal{B}(p,q,r)^s $  such that for all $n\in\mathbb{N}$, \begin{eqnarray}\label{eq:at21}\sum_{k\in\mathbb{Z}}\lambda_k\mathbb{E}_na^k=f_n \end{eqnarray} and for any $0<\eta\le 1$, \begin{equation}\label{eq:at22}\left\|\sum_{k\ge 0}\left(\frac{\lambda_k}{\|1_{B_{\nu^k}}\|_{p,q}}\right)^{\eta}\mathbf{1}_{B_{\nu^k}}\right\|^{\frac{1}{\eta}}_{\frac{p}{\eta},\frac{q}{\eta}}\le C\|f\|_{H^s_{p,q}}.\end{equation} Moreover, $$\sum_{k=l}^m\lambda_ka^k\longrightarrow f$$ in $H^s_{p,q}$  as $m\rightarrow\infty,\,\,l\rightarrow -\infty$. 
\vskip .1cm
Conversely if if $f\in \mathcal{M}$ has a decomposition as in (\ref{eq:at21}), then for any $0<\eta\le 1$, $$\|f\|_{H^s_{p,q}} \le C\left\|\sum_{k\in\mathbb{Z}}\left(\frac{\lambda_k}{\|1_{B_{\nu^k}}\|_{p,q}}\right)^{\eta}\mathbf{1}_{B_{\nu^k}}\right\|^{\frac{1}{\eta}}_{\frac{p}{\eta},\frac{q}{\eta}}.$$
\end{theorem}
For the last two spaces, we obtain the following two atomic decompositions.
\begin{theorem}\label{thm:at3}
Let $0<p<\infty$ and $0<q\leq \infty$. If the martingale $f\in \mathcal{M}$ is in $\mathcal{Q}_{p,q}$ (resp. $\mathcal{P}_{p,q}$), then there exists a sequence of triplets $(\lambda_k,a^k,\nu^k)\in \mathcal{A}(p,q,\infty)^S $ (resp. $\mathcal{A}(p,q,\infty)^*$) such that for any $n\in\mathbb{N}$, \begin{eqnarray}\label{eq31}\sum_{k\in\mathbb{Z}}\lambda_k\mathbb{E}_na^k=f_n \end{eqnarray} and for any $0<\eta\leq 1,$ \begin{eqnarray}\label{eq32}\left\|\sum_{k\in\mathbb{Z}}\left(\frac{\lambda_k}{\mathbb{P}(B_{\nu^k})^{\frac{1}{p}}}\right)^{\eta}\mathbf{1}_{B_{\nu^k}}\right\|^{\frac{1}{\eta}}_{\frac{p}{\eta},\frac{q}{\eta}}\le C\|f\|_{\mathcal{Q}_{p,q}} (\textrm{resp.}\,\|f\|_{\mathcal{P}_{p,q}}).\end{eqnarray} Moreover, $$\sum_{k=l}^m\lambda_ka^k\longrightarrow f$$ in $\mathcal{Q}_{p,q}$ (resp. $\mathcal{P}_{p,q}$) as $m\rightarrow\infty,\,\,l\rightarrow -\infty$.
\vskip .1cm
 Conversely, if $f\in \mathcal{M}$ has a decomposition as in (\ref{eq31}), then for any $0<\eta\le 1$, $$\|f\|_{\mathcal{Q}_{p,q}} (\textrm{resp.} \|f\|_{\mathcal{P}_{p,q}})\le C\left\|\sum_{k\in\mathbb{Z}}\left(\frac{\lambda_k}{\mathbb{P}(B_{\nu^k})^{\frac{1}{p}}}\right)^{\eta}\mathbf{1}_{B_{\nu^k}}\right\|^{\frac{1}{\eta}}_{\frac{p}{\eta},\frac{q}{\eta}}.$$
\end{theorem}

\begin{theorem}\label{thm:at4}
Let $0<p<\infty$ and $0<q\leq \infty$. If the martingale $f\in \mathcal{M}$ is in $\mathcal{Q}_{p,q}$ (resp. $\mathcal{P}_{p,q}$), then there exists a sequence of triplets $(\lambda_k,a^k,\nu^k)\in \mathcal{B}(p,q,\infty)^S $ (resp. $\mathcal{B}(p,q,\infty)^*$) such that for amy $n\in\mathbb{N}$, \begin{eqnarray}\label{eq41}\sum_{k\in\mathbb{Z}}\lambda_k\mathbb{E}_na^k=f_n \end{eqnarray} and for any $0<\eta\leq 1,$ \begin{eqnarray}\label{eq42}\left\|\sum_{k\in\mathbb{Z}}\left(\frac{\lambda_k}{\|1_{B_{\nu^k}}\|_{p,q}}\right)^{\eta}\mathbf{1}_{B_{\nu^k}}\right\|^{\frac{1}{\eta}}_{\frac{p}{\eta},\frac{q}{\eta}}\le C\|f\|_{\mathcal{Q}_{p,q}} (\textrm{resp.}\,\|f\|_{\mathcal{P}_{p,q}}).\end{eqnarray} Moreover, $$\sum_{k=l}^m\lambda_ka^k\longrightarrow f$$ in $\mathcal{Q}_{p,q}$ (resp. $\mathcal{P}_{p,q}$) as $m\rightarrow\infty,\,\,l\rightarrow -\infty$.
\vskip .1cm
 Conversely if if $f\in \mathcal{M}$ has a decomposition as in (\ref{eq41}), then for any $0<\eta\le 1$, $$\|f\|_{\mathcal{Q}_{p,q}} (\textrm{resp.} \|f\|_{\mathcal{P}_{p,q}})\le C\left\|\sum_{k\in\mathbb{Z}}\left(\frac{\lambda_k}{\|1_{B_{\nu^k}}\|_{p,q}}\right)^{\eta}\mathbf{1}_{B_{\nu^k}}\right\|^{\frac{1}{\eta}}_{\frac{p}{\eta},\frac{q}{\eta}}.$$
\end{theorem}
\vskip .3cm
Denote by $L_{2}^0$ the set of all $f\in L_2$ such that $\mathbb{E}_0f=0$. For $f\in L_2^0$, put $f_n=\mathbb{E}_nf$. We recall that $(f_n)_{n\geq 0}$ is in $\mathcal{M}$ and $L_2$-bounded. Moreover, $(f_n)_{n\geq 0}$ converges to $f$ in $L_2$ (see \cite{Neveu}).
\vskip .1cm
Define the function $\phi:\mathcal{F}\longrightarrow(0,\infty)$ by $$\phi(A)=\frac{\|\mathbf{1}_{A}\|_{p,q}}{\mathbb{P}(A)}$$ for all $A\in\mathcal{F}$, $P(A)\neq 0$.  We then define the Campanato space $\mathcal{L}_{2,\phi}$ as 
$$\mathcal{L}_{2,\phi}:=\left\{f\in L_2^0: \|f\|_{\mathcal{L}_{2,\phi}}:=\sup_{\nu\in \mathcal{T}}\frac{1}{\phi(B_\nu)}\left(\frac{1}{\mathbb{P}(B_\nu)}\int_{B_\nu}|f-f^\nu|^2\mathrm{d}\mathbb{P}\right)^{\frac{1}{2}}<\infty\right\}.$$
Our characterization of the dual space of $H_{p,q}^s$ spaces for $0<p\le q\le 1$ is as follows.
\begin{theorem}\label{thm:duality}
Let $0<p\le q\le1$. For $\kappa\in (H^s_{p,q})^*,$ the dual space of $H^s_{p,q}$, there exists $g\in\mathcal{L}_{2,\phi}$ such that $$\kappa(f)=\mathbb{E}[fg]\quad\mbox{for all}\quad f\in H^s_{p,q}$$ and $$\|g\|_{\mathcal{L}_{2,\phi}}\le c\|\kappa\|.$$ Conversely, let $g\in\mathcal{L}_{2,\phi}.$ Then the mapping $$\kappa_g(f) = \mathbb{E}[fg]=\int_\Omega fg\,d\mathbb{P},\quad\forall f\in L_2(\Omega)$$ can be extended to a continuously linear functional on $H^s_{p,q}$ such that $$\|\kappa\|\le c\|g\|_{\mathcal{L}_{2,\phi}}.$$
\end{theorem}

\section{Proof of Atomic Decompositions}
We prove here  the atomic decompositions of the spaces $H^s_{p,q}$,$\mathcal{Q}_{p,q}$ and $\mathcal{P}_{p,q}.$
\begin{proof}[Proof of Theorm \ref{thm:at1}]
We only present here the case of $H^s_{p,q}$. The proof of the atomic decomposition of the other spaces is obtained mutatis mutandis.
\vskip .1cm
Let $f$ be in $H^s_{p,q}$ and  define the stopping time as \begin{eqnarray}\label{st}\nu^k :=\inf\{n\in\mathbb{N}: s_{n+1}(f)>2^k\}.\end{eqnarray} It is clear that $(\nu^k)_{k\in \mathbb{Z}}$ is nonnegative and nondecreasing. Take $\lambda_k=2^{k+1}\mathbb{P}(\nu^k\ne\infty)^{\frac{1}{p}},$ and \begin{eqnarray}\label{eq3}a^k=\frac{f^{\nu^{k+1}}-f^{\nu^k}}{\lambda_k},\,\,\textrm{if}\,\,\lambda_k\ne0\quad\mbox{and}\quad a^k=0,\,\,\textrm{if}\,\,\lambda_k=0.\end{eqnarray} 
Recalling that $$d_nf^{\nu^k}=f_n^{\nu^k} - f_{n-1}^{\nu^k} = \sum_{m=0}^{n}\mathbf{1}_{\{\nu^k\ge m\}}d_mf - \sum_{m=0}^{n-1}\mathbf{1}_{\{\nu^k\ge m\}}d_mf = \mathbf{1}_{\{\nu^k\ge n\}}d_nf,$$ we obtain that \begin{eqnarray*} s(f^{\nu^k}) &=& \left(\sum_{n\in\mathbb{N}}\mathbb{E}_{n-1}|d_nf^{\nu^k}|^2 \right)^{\frac{1}{2}} = \left(\sum_{n\in\mathbb{N}}\mathbb{E}_{n-1}|\mathbf{1}_{\{\nu^k\ge n\}}d_nf|^2 \right)^{\frac{1}{2}}\\ &=&\left(\sum_{n\in\mathbb{N}}\mathbf{1}_{\{\nu^k\ge n\}}\mathbb{E}_{n-1}|d_nf|^2 \right)^{\frac{1}{2}} = \left(\sum_{n=0}^{\nu^k}\mathbb{E}_{n-1}|d_nf|^2 \right)^{\frac{1}{2}}\\ &=& s_{\nu^k}(f).\end{eqnarray*}  Thus  by the definition of our stopping time, $s(f^{\nu^k})=s_{\nu^k}(f)\le2^k$. Moreover, $$\sum_{k\in \mathbb{Z}}(f_n^{\nu^{k+1}}-f_n^{\nu^{k}})=f_n\,\,\,\textrm{a.e}.$$
It follows that $(f_n^{\nu^{k}})_{n\geq 0}$ is an $L_2$-bounded martingale and so is $(a_n^{\nu^{k}})_{n\geq 0}$. Consequently, the limit $$\lim_{n\rightarrow\infty}a_n^k$$
exists a.e. in $L_2$. Hence (\ref{eq:at11}) is satisfied. 
\vskip .1cm
Let us check that $a^k$ is an $(p,r)^s$-atom. We start by noting that on the set $\{n\le\nu^k\}$ we have \begin{equation}\label{eq:vanishatom} a_n^k=\frac{f^{\nu^{k+1}}_n-f^{\nu^k}_n}{\lambda_k}=\frac{f_n-f_n}{\lambda_k}=0
\end{equation} 
by definition of stopped martingales and thus $\mathbb{E}_na^k=0$ when $\nu^k\geq n$. Hence $a^k$ satisfies condition (a1) in Definition \ref{atdef1}.
\vskip .1cm
Also we note that equation (\ref{eq:vanishatom}) also implies that
$$\mathbf{1}_{\{\nu^k=\infty\}}[s(a^k)]^2\le \sum_{n\in\mathbb{N}}\mathbf{1}_{\{\nu^k\ge n\}}\mathbb{E}_{n-1}|d_na^k|^2=\sum_{n\in\mathbb{N}}\mathbf{1}_{\{\nu^k\ge n\}}\mathbb{E}_{n-1}|\mathbf{1}_{\{\nu^k\ge n\}}d_na^k|^2=0.$$
That is the support of $s(a^k)$ is contained in $B_{\nu^k}=\{\nu^k\neq \infty\}$.
\vskip .1cm
Observing that $$d_na^k=\frac{d_n(f^{\nu^{k+1}} -f^{\nu^k})}{\lambda_k}=\frac{(d_nf)\mathbf{1}_\{\nu^k<n\leq \nu^{k+1}\}}{\lambda_k},$$
we obtain the following
\begin{eqnarray*}
[s(a^k)]^2 & = & \sum_{n\in\mathbb{N}}\mathbb{E}_{n-1}|d_na^k|^2 
\leq\left(\frac{s_{\nu^{k+1}}(f)}{\lambda_k}\right)^2\le \left(\frac{2^{k+1}}{\lambda_k}\right)^2= \left(\mathbb{P}(B_{\nu^k})^{-\frac{1}{p}}\right)^2.
\end{eqnarray*}
Therefore $s(a^k)\le\mathbb{P}(\nu^k\ne\infty)^{-\frac{1}{p}}$ and as $s(a^k)=0$ outside $B_{\nu^k}$, we easily obtain that  $$\|s(a^k)\|_{r}\le \mathbb{P}(\nu^k\ne\infty)^{\frac 1r-\frac{1}{p}}.$$ Thus condition (a2) in the definition of an $(p,r)^s$-atom also holds. 
\vskip .1cm
We next check that $$\sum_{k}\lambda_ka^k\longrightarrow f\,\,\textrm{in}\,\,H^s_{p,q}.$$
As  $$\lambda_ka^k=f^{\nu^{k+1}}-f^{\nu^k},$$ we obtain that $$\sum_{k=l}^m\lambda_ka^k = \sum_{k=l}^m(f^{\nu^{k+1}}-f^{\nu^k})=f^{\nu^{m+1}}-f^{\nu^l}.$$ Hence we have that \begin{eqnarray}\label{eq4}f-\sum_{k=l}^m\lambda_ka^k=(f-f^{\nu^{m+1}}) + f^{\nu^l}.\end{eqnarray} Now by definition, $$\|f-f^{\nu^{m+1}}\|_{H^s_{p,q}}=\|s(f-f^{\nu^{m+1}})\|_{p,q}.$$ Thus for $\Omega_j$ as in the definition of amalgam spaces, 
\begin{eqnarray*}
\|s(f-f^{\nu^{m+1}})\|_{p,q} & = & \left[\sum_{j\in\mathbb{Z}}\left(\int_{\Omega}s^p(f-f^{\nu^{m+1}})\mathbf{1}_{\Omega_j}\mathrm{d}\mathbb{P}\right)^{\frac{q}{p}}\right]^{\frac{1}{q}}\\
& = & \left[\sum_{j\in\mathbb{Z}}\|s^p(f-f^{\nu^{m+1}})\mathbf{1}_{\Omega_j}\|^q_p\right]^{\frac{1}{q}}
\end{eqnarray*}
As $s^p(f-f^{\nu^{m+1}})\le s^p(f)$ and $$\int_{\Omega_j}s^p(f)\mathrm{d}\mathbb{P}<\infty,$$ it follows from the Dominated Convergence Theorem that $$\|s^p(f-f^{\nu^{m+1}})\mathbf{1}_{\Omega_j}\|_p\longrightarrow 0$$  as $m\longrightarrow\infty$. Hence as $\sum_{j\in\mathbb{Z}}\|s^p(f)\mathbf{1}_{\Omega_j}\|^q_p=\|f\|_{H_{p,q}^s}^q<\infty$, applying the Dominated Convergence Theorem for the sequence space $\ell_q$, we conclude that $$\|f-f^{\nu^{m+1}}\|_{H^s_{p,q}}=\|s(f-f^{\nu^{m+1}})\|_{p,q}\longrightarrow 0$$ as $m\longrightarrow \infty.$ 
\vskip .1cm
Also since $s(f^{\nu^k})\le2^k,$ we have that $$\|f^{\nu^l}\|_{H^s_{p,q}}=\|s(f^{\nu^l})\|_{p,q}\le 2^l.$$ Hence $\|f^{\nu^l}\|_{H^s_{p,q}}\longrightarrow0$ as $l\longrightarrow -\infty.$ Thus (\ref{eq4}) implies that
\begin{eqnarray*}
\left\|f-\sum_{k=l}^m\lambda_ka^k\right\|_{H^s_{p,q}} &=&\left\|(f-f^{\nu^{m+1}}) + f^{\nu^l}\right\|_{H^s_{p,q}}\\
&\le&\|(f-f^{\nu^{m+1}})\|_{H^s_{p,q}} + \|f^{\nu^l}\|_{H^s_{p,q}}\rightarrow 0
\end{eqnarray*}
as $m\longrightarrow\infty$ and $l\longrightarrow -\infty.$ Hence $$\sum_{k=l}^m\lambda_ka^k\longrightarrow f$$ in $H^s_{p,q}$ as $m\rightarrow\infty,\,\,l\rightarrow -\infty.$ 
\vskip .1cm
Let us now establish (\ref{eq:at12}). Let $\Omega_j\subset\Omega$ be as in the definition of amalgam space. Then by definition, $$\left\|\sum_{k\in\mathbb{Z}}\left(\frac{\lambda_k}{\mathbb{P}(B_{\nu^k})^{\frac{1}{p}}}\right)^{\eta}\mathbf{1}_{B_{\nu^k}}\right\|^{\frac{1}{\eta}}_{\frac{p}{\eta},\frac{q}{\eta}}=\left[\sum_{j\in\mathbb{Z}}\left(\int_{\Omega}\left(\sum_{k\in\mathbb{Z}}\left(\frac{\lambda_k}{\mathbb{P}(B_{\nu^k})^{\frac{1}{p}}}\right)^{\eta}\mathbf{1}_{B_{\nu^k}} \right)^{\frac{p}{\eta}}\mathbf{1}_{\Omega_j}\mathrm{d}\mathbb{P}\right)^{\frac{q}{p}}\right]^{\frac{\eta}{q}\cdot\frac{1}{\eta}}.$$ Considering the inner sum, we see that $$\sum_{k\in\mathbb{Z}}\left(\frac{\lambda_k}{\mathbb{P}(B_{\nu^k})^{\frac{1}{p}}}\right)^{\eta}\mathbf{1}_{B_{\nu^k}} =\sum_{k\in\mathbb{Z}}\left(\frac{2^{k+1}\mathbb{P}(\nu^k\ne\infty)^{\frac{1}{p}}}{\mathbb{P}(B_{\nu^k})^{\frac{1}{p}}}\right)^{\eta}\mathbf{1}_{B_{\nu^k}}  = \sum_{k\in\mathbb{Z}}\left(2^{k+1}\right)^{\eta}\mathbf{1}_{B_{\nu^k}}.$$ 
We shall borrow an idea from \cite[pp 21-22]{Xie}. Let $G_k=B_{\nu^k}\setminus B_{\nu^{k+1}}$ where $B_{\nu^k}=\{\nu^k\ne\infty\}.$ Then $G_k$ are disjoint such that $B_{\nu^k}=\bigcup_{r=k}^{\infty}G_r$ and \begin{eqnarray}\label{dj1}\mathbf{1}_{B_{\nu^k}}=\sum_{r=k}^{\infty}\mathbf{1}_{G_r}.\end{eqnarray} Hence 
\begin{eqnarray*}
\sum_{k\in\mathbb{Z}}\left(2^{k+1}\right)^{\eta}\mathbf{1}_{B_{\nu^k}} &=& \sum_{k\in\mathbb{Z}}2^{(k+1)\eta}\cdot\sum_{r=k}^{\infty}\mathbf{1}_{G_r}\\ &=& \sum_{r\in\mathbb{Z}}\sum_{k\le r}2^{(k+1)\eta}\mathbf{1}_{G_r}\\
&\le&\frac{2^\eta}{2^{\eta}-1}\sum_{k\in\mathbb{Z}}2^{(k+1)\eta}\mathbf{1}_{G_k}.
\end{eqnarray*}
Thus
$$\sum_{k\in\mathbb{Z}}\left(2^{k+1}\right)^{\eta}\mathbf{1}_{B_{\nu^k}} \le\frac{4^{\eta}}{2^{\eta}-1}\left(\sum_{k\in\mathbb{Z}}s(f)\mathbf{1}_{G_k}\right)^{\eta}=\frac{(4)^{\eta}s(f)^{\eta}}{2^{\eta}-1}\sum_{k\in\mathbb{Z}}\mathbf{1}_{G_k}.$$ 
It follows that 
\begin{eqnarray*}
\left\|\sum_{k\in\mathbb{Z}}\left(\frac{\lambda_k}{\mathbb{P}(B_{\nu^k})^{\frac{1}{p}}}\right)^{\eta}\mathbf{1}_{B_{\nu^k}}\right\|_{\frac{p}{\eta},\frac{q}{\eta}}&=&\left[\sum_{j\in\mathbb{Z}}\left(\int_{\Omega}\left(\sum_{k\in\mathbb{Z}}\left(\frac{\lambda_k}{\mathbb{P}(B_{\nu^k})^{\frac{1}{p}}}\right)^{\eta}\mathbf{1}_{B_{\nu^k}} \right)^{\frac{p}{\eta}}\mathbf{1}_{\Omega_j}\mathrm{d}\mathbb{P}\right)^{\frac{q}{p}}\right]^{\frac{\eta}{q}}\\
&\le& \left[\sum_{j\in\mathbb{Z}}\left(\int_{\Omega}\frac{4^ps(f)^p}{(2^{\eta}-1)^{\frac{p}{\eta}}}\sum_{k\in\mathbb{Z}}\mathbf{1}_{G_k}\mathbf{1}_{\Omega_j}\mathrm{d}\mathbb{P}\right)^{\frac{q}{p}}\right]^{\frac{\eta}{q}}\\
&=& \left[\sum_{j\in\mathbb{Z}}\left(\frac{4^p}{(2^{\eta}-1)^{\frac{p}{\eta}}}\sum_{k\in\mathbb{Z}}\int_{G_k}s^p(f)\mathbf{1}_{\Omega_j}\mathrm{d}\mathbb{P}\right)^{\frac{q}{p}}\right]^{\frac{\eta}{q}}\\
&\le& \left(\frac{4^\eta}{2^{\eta}-1}\right)\left[\sum_{j\in\mathbb{Z}}\left(\int_{\Omega}s^p(f)\mathbf{1}_{\Omega_j}\mathrm{d}\mathbb{P}\right)^{\frac{q}{p}}\right]^{\frac{\eta}{q}}.
\end{eqnarray*}
That is
\begin{eqnarray*}
\left\|\sum_{k\in\mathbb{Z}}\left(\frac{\lambda_j}{\mathbb{P}(B_{\nu^k})^{\frac{1}{p}}}\right)^{\eta}\mathbf{1}_{B_{\nu^k}}\right\|^{\frac{1}{\eta}}_{\frac{p}{\eta},\frac{q}{\eta}}&\le& \left(\frac{4^\eta}{2^{\eta}-1}\right)^{\frac 1{\eta }}\|s(f)\|_{p,q}=\left(\frac{4^\eta}{2^{\eta}-1}\right)^{\frac 1{\eta }}\|f\|_{H^s_{p,q}}.
\end{eqnarray*}
 The first part of the theorem is then established.
 \vskip .2cm
 Conversely, let the martingale $f$ have a representation as in (\ref{eq:at11}). Then as $ s(a^k)\le\mathbb{P}(\nu^k\ne\infty)^{-\frac{1}{p}}$ with support in $B_{\nu^k}$, we obtain that 
\begin{eqnarray*}
\|f\|_{H^s_{p,q}}&=&\|s(f)\|_{p,q}
\le \left\|\sum_{k\in\mathbb{Z}}\lambda_ks(a^k)\right\|_{p,q}\\ &\le& \left\|\sum_{k\in\mathbb{Z}}\lambda_k\mathbb{P}(B_{\nu^k})^{-\frac{1}{p}}\mathbf{1}_{B_{\nu^k}}\right\|_{p,q}\\ &=& \left\|\sum_{k\in\mathbb{Z}}\frac{\lambda_k}{\mathbb{P}(B_{\nu^k})^{\frac{1}{p}}}\mathbf{1}_{B_{\nu^k}}\right\|_{p,q}.
\end{eqnarray*}
 Let us quickly check that $$\left\|\sum_{k\in\mathbb{Z}}\frac{\lambda_k}{\mathbb{P}(B_{\nu^k})^{\frac{1}{p}}}\mathbf{1}_{B_{\nu^k}}\right\|_{p,q}\le \left\|\sum_{k\in\mathbb{Z}}\left(\frac{\lambda_k}{\mathbb{P}(B_{\nu^k})^{\frac{1}{p}}}\right)^{\eta}\mathbf{1}_{B_{\nu^k}}\right\|^{\frac{1}{\eta}}_{\frac{p}{\eta},\frac{q}{\eta}}$$ for $0<\eta<1.$ Indeed by definition, 
\begin{eqnarray*}
\left\|\sum_{k\in\mathbb{Z}}\frac{\lambda_k}{\mathbb{P}(B_{\nu^k})^{\frac{1}{p}}}\mathbf{1}_{B_{\nu^k}}\right\|_{p,q} &=&\left[\sum_{j\in\mathbb{Z}}\left(\int_{\Omega}\left(\sum_{k\in\mathbb{Z}}\frac{\lambda_k}{\mathbb{P}(B_{\nu^k})^{\frac{1}{p}}}\mathbf{1}_{B_{\nu^k}} \right)^p\mathbf{1}_{\Omega_j}\mathrm{d}\mathbb{P}\right)^{\frac{q}{p}}\right]^{\frac{1}{q}}\\
&=&\left[\sum_{j\in\mathbb{Z}}\left(\int_{\Omega}\left(\sum_{k\in\mathbb{Z}}\frac{\lambda_k}{\mathbb{P}(B_{\nu^k})^{\frac{1}{p}}}\mathbf{1}_{B_{\nu^k}} \right)^{p\frac{\eta}{\eta}}\mathbf{1}_{\Omega_j}\mathrm{d}\mathbb{P}\right)^{\frac{q/\eta}{p/\eta}}\right]^{\frac{1}{q}\cdot\frac{\eta}{\eta}}\\
&\le& \left[\sum_{j\in\mathbb{Z}}\left(\int_{\Omega}\left(\sum_{k\in\mathbb{Z}}\left(\frac{\lambda_k}{\mathbb{P}(B_{\nu^k})^{\frac{1}{p}}}\right)^{\eta}\mathbf{1}_{B_{\nu^k}} \right)^{\frac{p}{\eta}}\mathbf{1}_{\Omega_j}\mathrm{d}\mathbb{P}\right)^{\frac{q/\eta}{p/\eta}}\right]^{\frac{\eta}{q}\cdot\frac{1}{\eta}}\\
&=&\left\|\sum_{k\in\mathbb{Z}}\left(\frac{\lambda_k}{\mathbb{P}(B_{\nu^k})^{\frac{1}{p}}}\right)^{\eta}\mathbf{1}_{B_{\nu^k}}\right\|^{\frac{1}{\eta}}_{\frac{p}{\eta},\frac{q}{\eta}}.
\end{eqnarray*}
Hence $$\|f\|_{H^s_{p,q}}\le \left\|\sum_{k\in\mathbb{Z}}\left(\frac{\lambda_k}{\mathbb{P}(B_{\nu^k})^{\frac{1}{p}}}\right)^{\eta}\mathbf{1}_{B_{\nu^k}}\right\|^{\frac{1}{\eta}}_{\frac{p}{\eta},\frac{q}{\eta}}$$ establishing the converse. The theorem is proved.
\end{proof}
\begin{proof}[Proof of Theorem \ref{thm:at2}]
The proof of Theorem \ref{thm:at2} follows similarly. The main changes are as follows. Let $f$ be in $H^s_{p,q}$ and  define the stopping time \begin{eqnarray}\label{st}\nu^k :=\inf\{n\in\mathbb{N}: s_{n+1}(f)>2^k\}.\end{eqnarray} The sequence here is taken as  $\lambda_k=2^{k+1}\|\mathbf{1}_{B_{\nu^k}}\|_{p,q},$ and again \begin{eqnarray}\label{eq3}a^k=\frac{f^{\nu^{k+1}}-f^{\nu^k}}{\lambda_k},\,\,\textrm{if}\,\,\lambda_k\ne0\quad\mbox{and}\quad a^k=0,\,\,\textrm{if}\,\,\lambda_k=0.\end{eqnarray} 
As above one obtains that $$s(a^k)\le \|\mathbf{1}_{B_{\nu^k}}\|_{p,q}^{-1}$$
and $s(a^k)=0$ on $\{\nu^k=\infty\}$. It follows that $$\|s(a^k)\|_r\le \|\mathbf{1}_{B_{\nu^k}}\|_{p,q}^{-1}\mathbb{P}(B_{\nu^k})^{\frac 1r}.$$
The remaining of the proof follows as above.
\end{proof}
\begin{proof}[Proof of Theorem \ref{thm:at3}]
Let $f\in\mathcal{Q}_{p,q}$ (resp. $f\in\mathcal{P}_{p,q}$) Then there exists an adapted non-decreasing, non-negative sequence $(\beta_n)_{n\in\mathbb{N}}$ such that $$S_n(f)\le\beta_{n-1}\,\, (\textrm{resp.}\, |f_n|\le\beta_{n-1})$$ and $$\|\beta_{\infty}\|_{p,q}\le 2\|f\|_{\mathcal{Q}_{p,q}} (\textrm{resp.}\, \|f\|_{\mathcal{P}_{p,q}}).$$ As stopping time, we take \begin{eqnarray}\label{eq8}\nu^k:=\inf\{n\in\mathbb{N}: \beta_n>2^k\} \end{eqnarray} and define $\lambda_k=2^{k+2}\mathbb{P}(\nu^k\ne\infty)^{\frac{1}{p}},$ and \begin{eqnarray}\label{eq9}a^k=\frac{f^{\nu^{k+1}}-f^{\nu^j}}{\lambda_k}\,\,\textrm{if}\,\,\lambda_k\neq 0, \,\textrm{and}\,\, a^k=0\,\,\textrm{otherwise}.\end{eqnarray}
As in Theorem \ref{thm:at1}, we obtain that $a^k$ satisfies condition (a1) in the definition of $(p,\infty)^S$-atom (resp. $(p,\infty)^*$-atom). Also, $a_n^k=0$ on $\{\nu^k=\infty\}$ for all $n\geq 0$, and the support of $S(a^k)$ (resp. $(a^k)^*$) is contained in $B_{\nu^k}$.
\vskip .1cm
We have that 
$$S(f^{\nu^k})=S_{\nu^k}(f)\leq \beta_{\nu^k-1}\le 2^k\,\,(\textrm{resp.}\,\,(f^{\nu^k})^*\le \beta_{\nu^k-1}\le 2^k).$$
\vskip .1cm
We also obtain that
\begin{eqnarray*}
[S(a^k)]^2 &=& \sum_{n\geq 0}|d_na^{\nu^k}|^2 \le \sum_{n\geq 0}\left|\frac{(d_nf)\mathbf{1}_{\nu^k<n\leq \nu^{k+1}}}{\lambda_k}\right|^2\\ &\le& \left(\frac{S_{\nu^{k+1}}(f)}{\lambda_k}\right)^2\le \left(\frac{2^{k+1}}{\lambda_k}\right)^2\le \mathbb{P}(B_{\nu^k})^{-\frac 2p}
\end{eqnarray*}
Also $$\left(\textrm{resp.}\,(a^k)^*\le \frac{(f^{\nu^{k+1}})^*+(f^{\nu^{k}})^*}{\lambda_k}\le \frac{2^{k+2}}{\lambda_k}=\mathbb{P}(B_{\nu^k})^{-\frac 1p}\right).$$
Thus $\|S(a^k)\|_{\infty}\le\mathbb{P}(\nu^k\ne\infty)^{-\frac{1}{p}}$ (resp. $\|(a^k)^*\|_{\infty}\le\mathbb{P}(\nu^k\ne\infty)^{-\frac{1}{p}}$). Hence condition (a2) in the definition of an $(p,\infty)^S$-atom (resp. $(p,\infty)^*$-atom) is satisfied.

We next prove that $\sum_{k=l}^m\lambda_ka^k$ converges to $f$ in $\mathcal{Q}_{p,q}$ (resp. $\mathcal{P}_{p,q}$) as $l\rightarrow -\infty$ and $m\rightarrow\infty$. As usual, define $$\zeta^j_{n-1}=\mathbf{1}_{\{\nu^k\le n-1\}}\|S(a^k)\|_{\infty}\quad\mbox{and}\quad (\zeta_{n-1})^2=\sum_{k=m+1}^{\infty}\lambda_k^2(\zeta^k_{n-1})^2$$
$$(\textrm{resp.}\,\zeta^j_{n-1}=\mathbf{1}_{\{\nu^k\le n-1\}}\|(a^k)^*\|_{\infty}\quad\mbox{and}\quad \zeta_{n-1}=\sum_{k=m+1}^{\infty}\lambda_k\zeta^k_{n-1}).$$ Then we have (see \cite[p. 17]{Weisz})$$S_n(f-f^{\nu^{m+1}})\le \left(\sum_{k=m+1}^{\infty}\lambda_k^2(\zeta^k_{n-1})^2\right)^{\frac 12}= \zeta_{n-1}\,\,(\textrm{resp.}\,|f_n-f_n^{\nu^{m}}|\le  \zeta_{n-1}).$$ 
Putting $T(a^k)=S(a^k),(a^k)^*$, we obtain
$$\zeta_{n-1}\le \sum_{k=m+1}^{\infty}\lambda_k\zeta^k_{n-1}\le \sum_{k=m+1}^{\infty}\lambda_k\|T(a^k)\|_{\infty}\mathbf{1}_{\{\nu^k\le n-1\}}\le \sum_{k=m+1}^{\infty}\frac{\lambda_k}{P(B_{\nu^k})^{\frac 1p}}\mathbf{1}_{B_{\nu^k}}.$$
It follows that
$$S(f-f^{\nu^{m+1}})\,(\textrm{resp.}\, (f-f^{\nu^{m+1}})^*)\leq \lim_{n\rightarrow\infty}\zeta_n\le \sum_{k=m+1}^{\infty}\frac{\lambda_k}{P(B_{\nu^k})^{\frac 1p}}\mathbf{1}_{B_{\nu^k}}=\sum_{k=m+1}^{\infty}2^{k+2}\mathbf{1}_{B_{\nu^k}}.$$
Hence 
$$\|f-f^{\nu^k{m+1}}\|_{\mathcal{Q}_{p,q}}^q (\textrm{resp.}\,\|f-f^{\nu^k{m+1}}\|_{\mathcal{P}_{p,q}})^q\le\|\zeta_\infty\|_{p,q}^q\le\sum_{j\in \mathbb{Z}}\left\|\left(\sum_{k=m+1}^{\infty}2^{k+2}\mathbf{1}_{B_{\nu^k}}\right)\mathbf{1}_{\Omega_j}\right\|_{p}^q.$$
Proceeding as in the proof of Theorem \ref{thm:at1}, we obtain that 
$$\left(\sum_{k=m+1}^{\infty}2^{k+2}\mathbf{1}_{B_{\nu^k}}\right)\mathbf{1}_{\Omega_j}\le C\beta_\infty\mathbf{1}_{\Omega_j}.$$
Hence as $\|\beta_\infty\mathbf{1}_{\Omega_j}\|_p<\infty$, it follows from the the Dominated Convergence Theorem that
$$ \left\|\left(\sum_{k=m+1}^{\infty}2^{k+2}\mathbf{1}_{B_{\nu^k}}\right)\mathbf{1}_{\Omega_j}\right\|_{p}\rightarrow 0\,\,\textrm{as}\,\, m\rightarrow\infty.$$
As $$\left\|\left(\sum_{k=m+1}^{\infty}2^{k+2}\mathbf{1}_{B_{\nu^k}}\right)\mathbf{1}_{\Omega_j}\right\|_{p}\le C\|\beta_\infty\mathbf{1}_{\Omega_j}\|_p$$
and as $$\sum_{j\in \mathbb{Z}}\|\beta_\infty\mathbf{1}_{\Omega_j}\|_p^q=\|\beta_\infty\|_{p,q}^q<\infty,$$ an application of the Dominated Convergence Theorem for sequence spaces leads to
$$\sum_{j\in \mathbb{Z}}\left\|\left(\sum_{k=m+1}^{\infty}2^{k+2}\mathbf{1}_{B_{\nu^k}}\right)\mathbf{1}_{\Omega_j}\right\|_{p}^q\rightarrow 0\,\,\textrm{as}\,\, m\rightarrow\infty.$$
Thus $\|f-f^{\nu^k{m+1}}\|_{\mathcal{Q}_{p,q}}\left(\textrm{resp.}\,\|f-f^{\nu^k{m+1}}\|_{\mathcal{P}_{p,q}}\right)\rightarrow 0\,\,\textrm{as}\,\, m\rightarrow\infty.$
\vskip .1cm
Similarly, we obtain that $\|f^{\nu^l}\|_{\mathcal{Q}_{p,q}}\,(\textrm{resp.}\, \|f^{\nu^l}\|_{\mathcal{P}_{p,q}})\longrightarrow0$ as $l\longrightarrow -\infty.$ Therefore
$$
\|f-\sum_{k=l}^m\lambda_ka^k\|_{\mathcal{Q}_{p,q}}\,(\textrm{resp.}\,\|f-\sum_{k=l}^m\lambda_ka^k\|_{\mathcal{P}_{p,q}})\to 0
$$
as $m\to\infty$ and $l\to -\infty.$ Hence $$\sum_{k=l}^m\lambda_ka^k\longrightarrow f$$ in $\mathcal{Q}_{p,q}$ (resp. $\mathcal{P}_{p,q}$) as $m\rightarrow\infty,\,\,l\rightarrow -\infty$ and thus for all $n\in\mathbb{N},$ $$\sum_{k\in\mathbb{Z}}\lambda_k\mathbb{E}_na^k=f_n.$$
Now, as in Theorem \ref{thm:at1}, we obtain
\begin{eqnarray*}
\left\|\sum_{k\in\mathbb{Z}}\left(\frac{\lambda_k}{\mathbb{P}(B_{\nu^k})^{\frac{1}{p}}}\right)^{\eta}\mathbf{1}_{B_{\nu^k}}\right\|^{\frac{1}{\eta}}_{\frac{p}{\eta},\frac{q}{\eta}}&\le& \left\|\sum_{k\in\mathbb{Z}}(2^{k+2})^{\eta}\mathbf{1}_{B_{\nu^k}}\right\|^{\frac{1}{\eta}}_{\frac{p}{\eta},\frac{q}{\eta}}\\
&\le& C\left\|\beta_{\infty}\right\|_{p,q}\\ &\le& 2C\|f\|_{\mathcal{Q}_{p,q}}\,(\textrm{resp.}\,\|f\|_{\mathcal{P}_{p,q}}).
\end{eqnarray*} 

Conversely, assume that $f\in \mathcal{M}$ has the decomposition (\ref{eq31}). Define $\beta_n$ by $$\beta_n:=\sum_{k\in\mathbb{Z}}\lambda_k\|S(a^k)\|_{\infty}\mathbf{1}_{\{\nu^k\le n\}}\,\,\left(\textrm{resp.}\,\beta_n:=\sum_{k\in\mathbb{Z}}\lambda_k\|(a^k)^*\|_{\infty}\mathbf{1}_{\{\nu^k\le n\}}\right).$$
Then $(\beta_n)_{n\ge}$ is a nondecreasing nonnegative adapted sequence also, for $n\ge 0$, $$S_n(f)\le \beta_{n-1}\,\,(\textrm{resp.}\,|f_n|\le \beta_{n-1}).$$ 
As $\|S(a^k)\|_{\infty}\,\,(\textrm{resp.}\,\|(a^k)^*\|_{\infty})\le\mathbb{P}(\nu^k\ne\infty)^{-\frac{1}{p}}$, it follows that $$\|\beta_{\infty}\|_{p,q}\le\left\|\sum_{k\in\mathbb{Z}}\frac{\lambda_k}{\mathbb{P}(B_{\nu^k})^{\frac{1}{p}}}\mathbf{1}_{B_{\nu^k}} \right\|_{p,q}\le \left\|\sum_{k\in\mathbb{Z}}\left(\frac{\lambda_k}{\mathbb{P}(B_{\nu^k})^{\frac{1}{p}}}\right)^{\eta}\mathbf{1}_{B_{\nu^k}}\right\|^{\frac{1}{\eta}}_{\frac{p}{\eta},\frac{q}{\eta}}.$$ Thus  $$\|f\|_{\mathcal{Q}_{p,q}}\,\,(\textrm{resp.}\,\|f\|_{\mathcal{P}_{p,q}})\le\|\beta_\infty\|_{{p,q}}\le \left\|\sum_{k\in\mathbb{Z}}\left(\frac{\lambda_k}{\mathbb{P}(B_{\nu^k})^{\frac{1}{p}}}\right)^{\eta}\mathbf{1}_{B_{\nu^k}}\right\|^{\frac{1}{\eta}}_{\frac{p}{\eta},\frac{q}{\eta}}.$$ The proof is complete.
\end{proof}
\vskip .1cm
The proof of Theorem \ref{thm:at4} follows similarly. We leave it to the interested reader.
\section{Proof of the duality result}
We start this section by introducing the following result which is essentially \cite[Proposition 2.1]{Justin2}. We provide a proof for a continuous reading.
\begin{proposition}\label{prop2}
Let $0<p<1$ and $0<q\le1.$ For all finite sequence $\{f_n\}_{n=-m}^m$ of elements in $L_{p,q}(\Omega),$ we have $$\sum_{n=-m}^m\|f_n\|_{p,q}\le\left\|\sum_{n=-m}^m|f_n|\right\|_{p,q}.$$
\end{proposition}
\begin{proof}
Let $0<p<1,\,\,0<q\le1$ and le $\{f_n\}_{n=0}^m$ be a finite sequence of elements of $L_{p,q}(\Omega).$ For $q=1,$ using the reverse Minkowski's inequality in $L_p$ (see \cite[p. 11-12]{Grafakos}), we obtain $$\sum_{n=-m}^m\|f_n\|_{p,1}=\sum_{j\in\mathbb{Z}}\left\|\sum_{n=-m}^m|f_n\mathbf{1}_{\Omega_j}|\right\|_p\le\left\|\sum_{n=-m}^m|f_n|\right\|_{p,1}.$$ Now assume that $0<q<1$ and set $$x_n:=\left\{\|f_n\mathbf{1}_{\Omega_j}\|_p\right\}_{j\in\mathbb{Z}}\quad\forall n=-m,\ldots,m.$$ Applying the reverse Minkowski's inequality in $\ell^q$ and $L_p$, we obtain \begin{eqnarray*}\sum_{n=-m}^m\|f_n\|_{p,q} &=& \sum_{n=-m}^m\|x_n\|_{\ell^q}\le\left\|\left\{\sum_{n=-m}^m\|f_n\mathbf{1}_{\Omega_j}\|_p\right\}_{j\in\mathbb{Z}}\right\|_{\ell^q}\\ &\le&\left\|\left\{\left\|\sum_{n=-m}^mf_n\mathbf{1}_{\Omega_j}\right\|_p\right\}_{j\in\mathbb{Z}}\right\|_{\ell^q}=\left\|\sum_{n=-m}^m|f_n|\right\|_{p,q}.\end{eqnarray*} 
\end{proof}
We next prove our first duality result.
\begin{proof}[Proof of Theorem \ref{thm:duality}]
Let us start by defining some spaces. For $\nu$ a stopping time, we define
$$L_2^\nu(\Omega):=\{f\in L_2(\Omega):\,\,\mathbb{E}_n(f)=0,\,\,\textrm{for}\,\,\nu\ge n,\, n\in \mathbb{N}\}$$
and 
$$L_2^\nu(B_{\nu}):=\{f\in L_2^\nu(\Omega):\textrm{supp}(f)\subseteq B_\nu\}.$$
We endow $L_2^\nu(B_{\nu})$ with $$\|f\|_{L_2^\nu(B_{\nu})}:=\left(\int_{B_\nu}|f|^2d\mathbb{P}\right)^{\frac 12}<\infty.$$
We will first prove that any continuous linear functional on $H_{p,q}^s(\Omega)$ is also continuous on $L_2^\nu(B_{\nu})$. 
\vskip .1cm
Let $f\in L_2^\nu(B_{\nu})\setminus\{0\}$ and consider $$a(\omega):=C\|f\|_{L_2^\nu(B_{\nu})}^{-1}\mathbb{P}(B_\nu)^{\frac 12-\frac 1p}f(\omega),\quad \omega\in \Omega.$$
Then for an appropriate choice of the constant (for a choice of the constant, use \cite[Proposition 2.6 and Theorem 2.11]{Weisz}), $a$ is an $(p,2)^s$-atom associated to $\nu\in \mathcal{T}$. Observing with Theorem \ref{thm:at1} that $$\|a\|_{H_{p,q}^s(\Omega)}\lesssim \|\mathbf{1}_{B_\nu}\|_{p,q}\mathbb{P}(B_\nu)^{-\frac 1p},$$ 
and recalling that $p\le q$, we obtain 
\begin{eqnarray*}
\|f\|_{H_{p,q}^s(\Omega)} &=& C^{-1}\|f\|_{L_2^\nu(B_{\nu})}\mathbb{P}(B_\nu)^{\frac 1p-\frac 12}\|a\|_{H_{p,q}^s(\Omega)}\\ &\lesssim& \|f\|_{L_2^\nu(B_{\nu})}\mathbb{P}(B_\nu)^{\frac 1p-\frac 12}\|\mathbf{1}_{B_\nu}\|_{p,q}\mathbb{P}(B_\nu)^{-\frac 1p}\\ &\lesssim& \|f\|_{L_2^\nu(B_{\nu})}\mathbb{P}(B_\nu)^{\frac 1p-\frac 12}.
\end{eqnarray*}
It follows that for any continuous linear functional $\kappa$ on $H_{p,q}^s(\Omega)$ with operator norm $\|\kappa\|$,
$$|\kappa(f)|\le \|\kappa\|\|f\|_{H_{p,q}^s(\Omega)}\lesssim \|\kappa\|\mathbb{P}(B_\nu)^{\frac 1p-\frac 12}\|f\|_{L_2^\nu(B_{\nu})}.$$
Hence $\kappa$ is continuous on $L_2^\nu(B_{\nu})$ with operator norm $$\|\kappa\|_{(L_2^\nu(B_{\nu}))^*}:=\sup_{\substack {f\in L_2^\nu(B_{\nu})\\ \|f\|_{L_2^\nu(B_{\nu})}\le 1}}\|\kappa(f)\|\lesssim \mathbb{P}(B_\nu)^{\frac 1p-\frac 12}\|\kappa\|.$$
As $L_2^\nu(B_{\nu})$ is a subspace of $L_2(B_{\nu})=L_2(B_{\nu},d\mathbb{P})$, it follows from the above observation and the Hahn-Banach Theorem that any continuous linear functional $\kappa$ on $H_{p,q}^s(\Omega)$ can be extended to a continuous linear functional $\kappa_\nu$ on $L_2(B_{\nu})$. As $L_2(B_{\nu})$ is auto-dual, it follows that there exists $g\in L_2(B_{\nu})$ such that $$\kappa_\nu(f)=\int_{B_\nu}fg\,d\mathbb{P},\quad\forall f\in L_2(B_{\nu}).$$
Consequently, $$\kappa(f)=\kappa_\nu(f)=\int_{B_\nu}fg\,d\mathbb{P},\quad\forall f\in L_2^\nu (B_{\nu}).$$ 
Next, as $L_2(\Omega)$ is a dense in $H_{p,q}^s(\Omega)$ (this follows from the fact that $p\le q<2$ and Theorem \ref{thm:at1}), we have that any element $\kappa$ of the dual space of $H_{p,q}^s(\Omega)$ can be represented by 
\begin{equation}\label{eq:dualL2}\kappa(f)=\int_{\Omega}fg\,d\mathbb{P},\quad\forall f\in L_2 (\Omega).\end{equation}  
We are going to prove that the function $g$ in (\ref{eq:dualL2}) is in $\mathcal{L}_{2,\phi}(\Omega)$.
\vskip .1cm
Let $\nu\in \mathcal{T}$ and let $f\in L_2^\nu (B_{\nu})$ with $\|f\|_{L_2^\nu (B_{\nu})}\le 1$. Define
$$a(\omega)=C\mathbb{P}(B_\nu)^{\frac 12-\frac 1p}\frac{(f-f^\nu)\mathbf{1}_{B_\nu}(\omega)}{\|(f-f^\nu)\mathbf{1}_{B_\nu}\|_{L_2(\Omega)}},\quad \omega\in \Omega.$$
Then for an appropriate choice of the constant, $a$ is an $(p,2)^s$-atom associated to the stopping time $\nu$ and $a\in L_2(\Omega)$. Hence $$\kappa(a)=\int_{\Omega}ag\,d\mathbb{P}=\int_{B_\nu}ag\,d\mathbb{P}.$$
Thus 
\begin{eqnarray*}
\left|\int_{B_\nu}a(g-g^\nu)\,d\mathbb{P}\right| &=& \left|\int_{B_\nu}ag\,d\mathbb{P}\right|\\ &=& |\kappa(a)|\\ &\le& \|\kappa\|\|a\|_{H_{p,q}^s(\Omega)}\\ &\lesssim& \|\kappa\|\|\mathbf{1}_{B_\nu}\|_{p,q}\mathbb{P}(B_\nu)^{-\frac 1p}.
\end{eqnarray*}
Hence
\begin{eqnarray*}
\left|\int_{B_\nu}f(g-g^\nu)\,d\mathbb{P}\right| &=& \left|\int_{B_\nu}(f-f^\nu)(g-g^\nu)\,d\mathbb{P}\right|\\ &\lesssim&  C^{-1}\|(f-f^\nu)\mathbf{1}_{B_\nu}\|_{L_2(\Omega)}\mathbb{P}(B_\nu)^{\frac 1p-\frac 12}\|\kappa\|\|\mathbf{1}_{B_\nu}\|_{p,q}\mathbb{P}(B_\nu)^{-\frac 1p}\\ &\lesssim& \mathbb{P}(B_\nu)^{-\frac 12}\|\mathbf{1}_{B_\nu}\|_{p,q}\|\kappa\|.
\end{eqnarray*}
Thus 
\begin{eqnarray*}
\left(\int_{B_\nu}|g-g^\nu|^2\,d\mathbb{P}\right)^{\frac 12} &:=& \sup_{\substack{f\in L_2^\nu(B_\nu)\\ \|f\|_{L_2^\nu(B_\nu)}\le 1}}\left|\int_{B_\nu}f(g-g^\nu)\,d\mathbb{P}\right|\\ &\lesssim& \mathbb{P}(B_\nu)^{-\frac 12}\|\mathbf{1}_{B_\nu}\|_{p,q}\|\kappa\|.
\end{eqnarray*}
This gives us 
$$\frac{1}{\phi(B_\nu)}\left(\frac{1}{\mathbb{P}(B_\nu)}\int_{B_\nu}|g-g^\nu|^2\,d\mathbb{P}\right)^{\frac 12}\lesssim \|\kappa\|,\quad\forall \nu\in\mathcal{T}.$$
Hence $g\in \mathcal{L}_{2,\phi}(\Omega)$, and the proof of the first part of the theorem is complete.
\vskip .2cm
Conversely, let $g\in\mathcal{L}_{2,\phi}(\Omega).$ Let $f\in H^s_{p,q}(\Omega).$ We know that for the stopping times $$\nu^k:=\inf\{n\in\mathbb{N}:s_{n+1}(f)>2^k\},\,k\in \mathbb{Z},$$ \begin{eqnarray}\label{eq2}\left\|\sum_{k\in\mathbb{Z}}\left(\frac{\lambda_k}{\mathbb{P}(B_{\nu^k})^{\frac{1}{p}}}\right)^{\eta}\mathbf{1}_{B_{\nu^k}}\right\|^{\frac{1}{\eta}}_{\frac{p}{\eta},\frac{q}{\eta}}\le C\|f\|_{H^s_{p,q}}.\end{eqnarray} and moreover, $$\sum_{k=l}^m\lambda_ka^k\longrightarrow f$$ in $H^s_{p,q}$ as $m\rightarrow\infty,\,\,l\rightarrow -\infty,$ where $(\lambda_k,a^k,\nu_k)\in \mathcal{A}(p,q,2)^s$.  Also since $a^k$ is $L_2$-bounded, for $f\in H^s_{p,q}(\Omega),$ $$\kappa_g(f)=\mathbb{E}[fg]=\sum_{k\ge0}\mathbb{E}[a^kg]$$ is well defined and linear.  Using this, Schwartz's inequality, and the fact that $\|s(a^k)\|_2\le\mathbb{P}(B_{\nu^k})^{\frac 12-\frac{1}{p}}$, we obtain
\begin{eqnarray*}
|\kappa_g(f)| &\le& \sum_{k\in \mathbb{Z}}\lambda_k\left|\int_\Omega a^k(g-g^{\nu^k})\mathrm{d}\mathbb{P}\right|\le\sum_{k\in \mathbb{Z}}\lambda_k\|a^k\|_2\left(\int_{B_{\nu^k}} |g-g^{\nu^k}|^2\mathrm{d}\mathbb{P}\right)^{\frac{1}{2}}\\
&\lesssim& \sum_{k\in \mathbb{Z}}\lambda_k\|s(a^k)\|_2\left(\int_{B_{\nu^k}} |g-g^{\nu^k}|^2\mathrm{d}\mathbb{P}\right)^{\frac{1}{2}}\\
&=& \sum_{k\in \mathbb{Z}}\lambda_k\frac{\|\mathbf{1}_{B_{\nu^k}}\|_{p,q}}{\mathbb{P}(B_{\nu^k})^{\frac{1}{p}}}\frac{1}{\phi(B_{\nu^k})}\left(\frac{1}{\mathbb{P}(B_{\nu^k})}\int |g-g^{\nu^k}|^2\mathrm{d}\mathbb{P}\right)^{\frac{1}{2}}.
\end{eqnarray*}
Hence using Proposition \ref{prop2}, we deduce that
\begin{eqnarray*}|\kappa_g(f)| &\lesssim& \|g\|_{\mathcal{L}_{2,\phi}}\sum_{k\in \mathbb{Z}}\left\|\frac{\lambda_k}{\mathbb{P}(B_{\nu^k})^{\frac{1}{p}}}\mathbf{1}_{B_{\nu^k}}\right\|_{p,q}\\ &\le& \|g\|_{2,\phi}\left\|\sum_{k\ge 0}\left(\frac{\lambda_k}{\mathbb{P}(B_{\nu^k})^{\frac{1}{p}}}\right)^{\eta}\mathbf{1}_{B_{\nu^k}}\right\|^{\frac{1}{\eta}}_{\frac{p}{\eta},\frac{q}{\eta}}\\ &\lesssim& \|f\|_{H^s_{p,q}}\|g\|_{2,\phi}.
\end{eqnarray*}
Thus $\kappa_g(f) = \mathbb{E}[fg]$ extends continuously on $H^s_{p,q}(\Omega)$ and the proof is complete.
\end{proof}

\section{Concluding comments}
We start this section by introducing two additional martingale Hardy-amalgam spaces. These spaces which we denote $H^S_{p,q}$ and $H^*_{p,q}$ are defined as follows
\begin{itemize}
\item[i.]$H^S_{p,q}(\Omega) = \{f\in\mathcal{M} : \|f\|_{H^S_{p,q}(\Omega)} :=\|S(f)\|_{p,q}<\infty\}$.
\item[iii.]$H^*_{p,q}(\Omega) = \{f : \|f\|_{H^*_{p,q}(\Omega)}:=\|f^*\|_{p,q}<\infty\}.$
\end{itemize}
We note that when $p=q$, the spaces $H^S_{p,q}$ and $H^*_{p,q}$ correspond respectively to the classical martingale Hardy spaces $H^S_{p}$ and $H^*_{p}$ discussed in \cite{Weisz}.
We note that the question of atomic decompositions of the spaces $H^S_p$ and $H^*_p,$ is still open. Nevertheless, embedding relations between the five classical martingale spaces are well known. Indeed we have the follows (see \cite[Theorem 2.11]{Weisz}).

\begin{proposition}\label{prop:weisz}
For any $f\in \mathcal{M}$, the following hold.
\begin{itemize}
\item[(i)] $\|f\|_{H^*_{p}}\le C_p\|f\|_{H^s_{p}},\quad \|f\|_{H^S_{p}}\le C_p\|f\|_{H^s_{p}}\qquad (0<p\le2)$
\item[(ii)] $\|f\|_{H^s_{p}}\le C_p\|f\|_{H^*_{p}},\quad \|f\|_{H^s_{p}}\le C_p\|f\|_{H^S_{p}}\qquad (2\le p<\infty)$
\item[(iii)] $\|f\|_{H^*_{p}}\le C_p\|f\|_{\mathcal{P}_{p}},\quad \|f\|_{H^S_{p}}\le C_p\|f\|_{\mathcal{Q}_{p}}\qquad (0<p<\infty)$
\item[(iv)] $\|f\|_{H^*_{p}}\le C_p\|f\|_{\mathcal{Q}_{p}},\quad \|f\|_{H^S_{p}}\le C_p\|f\|_{\mathcal{P}_{p}}\qquad (0<p<\infty)$
\item[(v)] $\|f\|_{H^s_{p}}\le C_p\|f\|_{\mathcal{P}_{p}},\quad \|f\|_{H^s_{p}}\le C_p\|f\|_{\mathcal{Q}_{p}}\qquad (0<p<\infty)$.
\end{itemize}
Moreover, if the $(\mathcal{F})_{n\ge 0}$ is regular, the above five spaces are equivalent. [ c.f. \cite[Theorem 2.11]{Weisz} for the following. ]
\end{proposition}
The above result can be used to prove the Burkholder-Davis-Gundy's inequality \cite[Theorem 2.12]{Weisz} which provides an equivalence between the spaces $H^S_{p}$ and $H^*_{p}$.
\vskip .2cm
We can then state our first open problem in this setting.
\vskip .1cm
{\bf Question 1:} 
Does Proposition \ref{prop:weisz} extend to martingale Hardy-Amalgam spaces introduced in this paper? 
\vskip .1cm
If one assumes that for any $j\in \mathbb{Z}$ and any $n\ge 0$, $\Omega_j\in \mathcal{F}_n$, then it is not hard to prove that the above proposition extends to our setting mutandis mutatis. Unfortunately, this hypothesis is too restrictive.
\vskip .2cm
Our second open question about these new spaces is about the duality for large exponents.
\vskip .1cm
{\bf Question 2:} Is there a characterization of the dual space of the space $H^s_{p,q}$ for $\max\{p,q\}>1$?
\vskip .1cm
The dual space of $H^s_{p}$ for $1<p<\infty$ is described in \cite[Theorem 2.26]{Weisz}. It is not clear how this can be extended to the space $H^s_{p,q}$ even under further restrictions on the family $\{\Omega_j\}_{j\in \mathbb{Z}}$.
 
\section{Declarations}
The authors declare that they have no conflict of interest regarding this work.



\end{document}